%% file: DIRK.tex
\documentclass{amsart}

\usepackage{geometry}
\usepackage{bibentry}

\input{AERMacros}

\newcommand{\AJS}[1]{#1}

\nobibliography*

\begin{document}

\title[DIRK schemes]{For DIRK schemes $U=B$}
\title[DIRK schemes]{\AJS{Bochner stability for $B$-stable DIRK schemes}}

\author[A.~Ramirez]{Anthony E.~Ramirez}
\email[A.~Ramirez]{aramir13@utk.edu}
\author[A.J.~Salgado]{Abner J.~Salgado}
\email[A.J.~Salgado]{asalgad1@utk.edu}

\address{Department of Mathematics, University of Tennessee, Knoxville TN 37996 USA.}

\date{Draft version of \today}

\subjclass[2020]{65LO6, 65J08, 65M12, 65L20}

\keywords{
    diagonally implicit Runge--Kutta methods,
    $B$-stability,
    algebraic stability,
    abstract evolution equations,
    Gelfand triples,
    gradient flows,
    discrete energy dissipation}


\begin{abstract}
    \AJS{In \emph{[\bibentry{MR4673496}]} the notion of $U$-stability for Diagonally Implicit Runge-Kutta (DIRK) schemes was introduced. Here we establish the equivalence between $U$- and $B$- (algebraic) stability for two- and three-stage DIRK schemes, which then} provides suitable Bochner norm a priori estimates for $B$-stable methods. \AJS{As applications}, we first prove the convergence of $U$-stable DIRK schemes for linear coercive evolution problems under minimal regularity via energy estimates and compactness. Second, we apply these discretizations to gradient flows, which allow us to derive discrete local energy dissipation inequalities and provide \AJS{counterexamples that} demonstrate the limitations of stagewise energy monotonicity.
\end{abstract}

\maketitle

\section{Introduction}

In \cite{MR4673496} a notion of stability, called $U$-stability, for a class of diagonally implicit Runge-Kutta (DIRK) schemes was introduced. This notion seems particularly suitable for the discretization of evolution equations in Gelfand triples, as it reproduces the basic a priori estimates that the continuous problem possesses, and that merely follow from the structural assumptions of the ``spatial'' operator. No further regularity on the solutions is required, and the stability proof respects the Gelfand triple structure of the problem; see Definition~\ref{def:Ustability} for a precise statement and further discussion.

Stability notions for ODE integrators are plentiful in the literature. The notions of zero-, $A$-, $A(\theta)$-, $B$-, or $L$-stability immediately come to mind, and any attempt at a complete list is doomed to fail. For this reason, we are content with referring to the existing and vast literature; see \cite{MR3559553,MR1227985,MR2657217,MR2840298}. Many of these, however, seem better suited for ODEs or evolution equations with a different structure than the ones we are focused on here since, for instance, these only allow  one to obtain stability in $L^\infty(0,T;H)$ which is not enough for convergence of these schemes. We refer to Section~\ref{sec:Notation} for notation.


Out of all these notions, that of algebraic or $B$-stability seems relevant to our discussion, since it naturally fits problems where dissipation is present, and there seem to be some proofs, with different levels of generality and abstraction, of convergence for $B$-stable schemes when applied to dissipative problems. For example, in the context of ODEs with a one-sided Lipschitz condition, general convergence results have been proven, for instance, in \cite[Section IV.12]{MR2657217}. For dynamical systems where dissipativity is obtained via an inner product, in \cite{MR1293524} it is shown that an algebraically stable Runge-Kutta scheme guarantees the existence of an absorbing set and preserves dissipativity for any fixed time step. For gradient flows, \cite{MR3194795} shows that algebraically stable schemes monotonically decrease the discrete energy at each time step under a step-size restriction. In addition, \cite{MR4121887} studied the contractive properties of these schemes for convex gradient systems. For linear problems, \cite{crouzeix1975approximation} has shown convergence under minimal regularity for $B$-stable schemes via a spectral decomposition of the operator.

It is remarkable, however, that to our knowledge most works dealing with Runge-Kutta schemes applied to evolution problems in Gelfand triples are concerned with the derivation of error estimates under suitable smoothness assumptions; see, for instance, \cite{MR1284670,MR2329662}. One notable exception is \cite{MR2600543} where convergence under minimal regularity for stiffly accurate schemes is presented, under the assumption that the spatial operator is monotone.

This brings us to one of the main goals of this work. We show that, for DIRK schemes, the notions of $U$- and $B$- stability actually coincide. This implicitly provides a proof of some Bochner norm a priori estimates for $B$-stable DIRK schemes, those that follow from $U$-stability, and that we have not been able to locate elsewhere. On the other hand, this proves convergence of $U$-stable schemes in scenarios where $B$-stable ones are known to converge.

To illustrate and apply this interplay we, first, prove convergence under minimal regularity for a class of linear evolution problems in Gelfand triples. A result that, for $B$-stable schemes, seems to be part of the numerical folklore, but for which we have not been able to provide a precise reference; at least in the form we provide it here --- via energy estimates and compactness arguments. As a second application, since it is known that $B$-stable schemes converge to gradient flows \cite{MR2566077}, we complement this result with some new properties that follow from $U$-stability.

Our presentation, essentially, follows the program we outlined out above. We introduce notation and preliminaries in Section~\ref{sec:Notation}. The equivalence in question is described in Section~\ref{sec:equivalence}. The convergence for linear problems is presented in Section~\ref{sec:ConvLinear}; and, finally, gradient flows are discussed in Section~\ref{sec:GradFlows}.

As a final introductory remark we mention that we limit our discussion to two and three stage schemes since, as shown in \cite{MR458890}, there is little gain, from the consistency point of view, in considering more stages.

\section{Notation and preliminaries}
\label{sec:Notation}

Let us begin by introducing the notation and standing assumptions we shall adopt. First, the relation $A \lesssim B$ means that $A \leq c B$ for a nonessential constant $c$ that may change at each occurrence. $A \coloneqq B$ means equality by definition. For $r \in [1,\infty]$, its H\"older conjugate is denoted by $r'$.

We denote by $T>0$ some positive, final, time. For a Banach space $X$ and $p \in [1,\infty]$ the spaces of vector-valued integrable functions are denoted and normed the usual way, i.e.,
$ L^p(0,T;X)$. For $k \in \Natural_0$ and $\gamma \in (0,1]$, $C^{k,\gamma}([0,T];X)$ is also understood as expected. For simplicity, we shall write $C([0,T];X) \coloneqq C^{0,0}([0,T];X)$.

\subsection{Evolution problems}
To describe the class of evolution problems we are interested in, we assume that the so-called state space $V$ is Banach and separable, the pivot space $H$ is Hilbert and separable, with $V \hookrightarrow H \eqsim H' \hookrightarrow V'$. We call this structure a Gelfand triple. The inner product in $H$ will be denoted by $(\cdot,\cdot)$, whereas the duality pairing for $V$ is denoted by $\langle \cdot, \cdot \rangle$. We recall that in the Gelfand triple framework this duality can be understood as an extension of the inner product of $H$ in the sense that
\[
  \langle v,w \rangle = (v,w), \qquad \forall v \in H, \quad w \in V.
\]

The ``spatial'' operator $\mclA : V \to V'$ is assumed to be $p$-coercive, to satisfy a $q$-growth condition, and some condition for local solvability. To be precise, 
\begin{description}
  \item[$p$-coercivity] There are $p \in (1,\infty)$ and $\alpha >0$ such that
  \begin{equation}
  \label{eq:pcoercive}
    \alpha \nrm{ v }_V^p \leq \langle \mclA v, v \rangle, \qquad \forall v \in V.
  \end{equation}
  
  \item[$q$-growth] There are $q\geq p$ and $\AJS{\mclM} : \Real \to \Real$, that is nonnegative and increasing, such that
  \[
    \nrm{ \mclA v }_{V'} \leq \AJS{\mclM} \left( \nrm{ v }_H \right) \nrm{ v }_V^{p/q'}, \qquad \forall v\in V.
  \]
  
  \item[Local solvability] There is $\gamma_0 >0$ such that, for every $\gamma \in (0,\gamma_0)$ and all $F \in V'$, the equation
  \[
    v + \gamma \mclA v = F
  \]
  has a unique solution $v \in V$.
\end{description}
Having fixed the exponents $p$ and $q$ we define
\[
  \mclW \coloneqq \left\{ w \in L^p(0,T;V) \ \middle| \ \diff_t w \in L^{q'}(0,T;V') \right\},
\]
and recall that $\mclW \hookrightarrow C([0,T];H)$.

We are now ready to describe the class of problems we are interested in. Given $u_0 \in H$ and $f \in L^{p'}(0,T;V')$ we must find $u \in \mclW$ such that
\[
  u(0) = u_0,
\]
and
\begin{equation}
\label{eq:VarSolution}
  \int_0^T \langle \diff_t u(t) + \mclA u(t), v(t) \rangle \diff t = \int_0^T \langle f(t), v(t) \rangle \diff t, \qquad \forall v \in L^p(0,T;V).
\end{equation}
We refer the reader to \cite[Part II]{MR3014456} for a study of existence, uniqueness, and properties of solutions.

\AJS{
\subsubsection{Gradient flows}
A particular case of the construction we described above is that of gradient flows, which in our setting are understood as follows. Let $V = H \eqsim V'$, $f \equiv 0$, and $\mclA = \Diff \Phi$, where $\Phi : H \to \overline{\Real}$ is a convex, LSC, and coercive potential with effective domain $\mclD(\Phi) \subset H$ such that $\overline{\mclD(\Phi)}  = H$. In addition, we assume that $u_0 \in \mclD(\Phi)$. In this case, we may rewrite \eqref{eq:VarSolution} as an evolution variational inequality. We must find
\[
  u \in H^1(0,T;H) \coloneqq \left\{ w \in L^2(0,T; H) \ \middle| \ \diff_t w \in L^2(0,T;H) \right\},
\]
such that $u(0)= u_0$ and, for almost every $t \in (0,T)$,
\begin{equation}
\label{eq:EVI}
  \left( \diff_t u(t), u(t) - v \right) + \Phi(u(t)) - \Phi(v) \leq 0, \quad \forall v \in H. 
\end{equation}
We recall that one of the fundamental properties of this gradient flow is the \emph{energy identity}
\begin{equation}
\label{eq:EnergyDissipation}
  \left\| \diff_t u(t) \right\|_H^2 + \diff_t \Phi(u(t)) =0, \qquad \text{a.e. } t \in (0,T).
\end{equation}
We refer the reader to \cite{MR348562} and \cite[Lecture 12]{MR4886890} for thorough treatments on gradient flows, and to \cite{MR1737503,MR1377244} for their temporal discretization via the implicit Euler method, their a priori and a posteriori error analyses.}

\subsection{DIRK schemes}
\label{sub:DIRK}

Here we review notation and order conditions regarding DIRK schemes as well as relevant stability notions. We let $s \in \{2,3\}$ be the number of stages, $\msfA \in \Real^{s \times s}$ be lower triangular and with positive diagonal entries, $\bfb, \bfc \in \Real^s$. We organize these in the so-called Butcher table
\begin{equation}
\label{eq:ButcherTable}
  \begin{tabular}{c|c}
    $\bfc$ & $\msfA$\\
    \hline
     & $\bfb^\top$
  \end{tabular}
\end{equation}
We comment that, while DIRK schemes can be defined for arbitrary $s \in \mathbb{N}$, as \cite{MR458890} shows, very little is gained by considering a higher number of stages.

Let us now explain how a Runge-Kutta scheme is applied to \eqref{eq:VarSolution}. We begin by introducing a partition of $[0,T]$ of size $\mclN \in \Natural$:
\[
  0 = t_0 < t_1 < \cdots < t_\mclN = T.
\]
We set, for $n = 0, \ldots, \mclN-1$,
\[
  \tau_n = t_{n+1} - t_n,
\]
and denote $I_n = [t_n,t_{n+1}]$, for $n \in \{ 0, \ldots, \mclN-1 \}$. By $\bftau \to \bfzero$ we shall denote
\[
  \lim_{\mclN\to \infty} \max_{n=0}^{\mclN-1} \tau_n = 0.
\]

Our approximation then computes sequences $\{u_n\}_{n=1}^\mclN \subset V$ and $\{v_{n,i}\}_{i=1,n=0}^{s,\mclN-1} \subset V$ as follows. For $n = 0, \ldots, \mclN-1$ the \emph{stages} are computed, for $i=1,\ldots,s$, by
\begin{equation}
\label{eq:Stages}
  \frac1{\tau_n} \left( v_{n,i} - u_n, w \right) + \sum_{j=1}^i a_{i,j} \langle \mclA v_{n,j}, w \rangle = \sum_{j=1}^i a_{i,j}  \langle f_{n,j}, w \rangle, \qquad \forall w \in V.
\end{equation}
Then we advance via
\begin{equation}
\label{eq:Advance}
  \frac1{\tau_n} \left( u_{n+1} - u_n ,w \right) + \sum_{i=1}^s b_i \langle \mclA v_{n,i}, w \rangle = \sum_{i=1}^s b_i \langle f_{n,i}, w \rangle, \qquad \forall w \in V.
\end{equation}
Here the sequence $\{f_{n,i}\}_{i=1,n=0}^{s,\mclN-1} \subset V'$ is meant to be an approximation $f_{n,i} \approx f(t_n + c_i \tau_n)$ which is not necessarily meaningful. Below we shall provide suitable ways of approximating these point values.

\subsubsection{Order conditions}
Let $\bfone = [1, \ldots, 1]^\top \in \Real^s$. We recall that, for consistency of order one, we must have
\[
  \bfb^\top \bfone = 1, \qquad \msfA \bfone = \bfc.
\]
In addition to these, for consistency of order two, we require
\[
  \bfb^\top \bfc = \frac12.
\]
Finally, for consistency of order three all the conditions above are necessary, and
\[
  \bfb^\top \msfA \bfc = \frac16.
\]

\subsubsection{$U$-stability}

We now recall the main stability notion advanced in \cite{MR4673496}. Since it will be useful in what follows we recall that, from \eqref{eq:Stages}--\eqref{eq:Advance}, we may obtain; see \cite[Proposition 3.2]{MR4673496}, the following extrapolation identity:
\begin{equation}
\label{eq:Extrapolation}
  u_{n+1} = \left( 1- \bflambda^\top \bfone \right) u_n + \sum_{i=1}^s \lambda_i v_{n,i},
\end{equation}
where 
\begin{equation}
\label{eq:DefOfLambda}
  \bflambda^\top = \bfb^\top \msfA^{-1}.
\end{equation}

For $s=2$ we define
\begin{align*}
  \delta_1 &= \frac12(\lambda_1 + \lambda_2)(2 -\lambda_1 - \lambda_2) , &
  \delta_2 &= \frac12 \lambda_2(2 - \lambda_2) ,
  \\
  \nu_{11} &= a_{11} \left[ \lambda_1(1-\lambda_2) + \lambda_2(2-\lambda_2) \right] , &
  \nu_{22} &= a_{22} \lambda_2,
  \\
  \nu_{12} &= \lambda_2 \left[ a_{21} + a_{11}(\lambda_1 + \lambda_2-2) \right] ,
\end{align*}
and
\[
  \msfQ(2) \coloneqq
  \begin{bmatrix}
    \delta_1 & -\delta_1 - \frac{\nu_{12}}{2a_{11}} & \frac{\nu_{12}}{2a_{11}} \\
    -\delta_1 - \frac{\nu_{12}}{2a_{11}} & \delta_1 + \delta_2 + \frac{\nu_{12}}{a_{11}} & - \delta_2 - \frac{\nu_{12}}{2a_{11}} \\
    \frac{\nu_{12}}{2a_{11}} & - \delta_2 - \frac{\nu_{12}}{2a_{11}} & \delta_2
  \end{bmatrix} \in \Real^{3 \times 3} ;
\]
whereas, when $s=3$,
\begin{align*}
    \delta_1 &= \frac12(\lambda_1 + \lambda_2 + \lambda_3)(2-\lambda_1 - \lambda_2 - \lambda_3) , \\
    \delta_2 &= \frac12(\lambda_2 + \lambda_3)(2-\lambda_2 - \lambda_3),
    \\
    \delta_3 &= \frac12\lambda_3 (2-\lambda_3), \\
    \nu_{11} &=  a_{11} (\lambda_1 (1 - \lambda_2 - \lambda_3) + (2 - \lambda_2 - \lambda_3) (\lambda_2 + \lambda_3)),
    \\
    \nu_{22} &=  a_{22} (\lambda_2 (1 - \lambda_3) + (2 - \lambda_3) \lambda_3), \\
    \nu_{33} &= a_{33}\lambda_3,
    \\
    \nu_{12} &=  [a_{21} (\lambda_2 (1- \lambda_3) + (2 - \lambda_3) \lambda_3)+ a_{11} (\lambda_2 (-2 + \lambda_1 + \lambda_2) + 2 (-1 + \lambda_2) \lambda_3 + \lambda_3^2)], \\
    \nu_{13} &=  \lambda_3 (a_{31} + a_{11} \lambda_1 + a_{21} (-2 + \lambda_2 + \lambda_3)),
    \\
    \nu_{23} &=  \lambda_3 (a_{32} + a_{22} (-2 + \lambda_2 + \lambda_3)),
\end{align*}
and the matrix $\msfQ(3) \in \Real^{4 \times 4}$ \AJS{is defined} via
\[
  \msfQ(3) \coloneqq 
  \begin{bmatrix}
    \delta_1 & -\delta_1 - \frac{\nu_{12}}{2a_{11}} - \frac{\nu_{13}}{2a_{11}} & \frac{\nu_{12}}{2a_{11}} - \frac{\nu_{23}}{2a_{22}} + \frac{a_{21}\nu_{23}}{2a_{11}a_{22}} & \frac{\nu_{13}}{2a_{11}} + \frac{\nu_{23}}{2a_{22}} - \frac{a_{21}\nu_{23}}{2a_{11}a_{22}}  \\
    -\delta_1 - \frac{\nu_{12}}{2a_{11}} - \frac{\nu_{13}}{2a_{11}}  & \delta_1 + \delta_2 + \frac{ \nu_{12}}{a_{11}} + \frac{ \nu_{13}}{a_{11}} & -\delta_2 - \frac{\nu_{12}}{2a_{11}} - \frac{a_{21}\nu_{23}}{2a_{11}a_{22}} & - \frac{\nu_{13}}{2a_{11}} + \frac{a_{21}\nu_{23}}{2a_{11}a_{22}} \\
    \frac{\nu_{12}}{2a_{11}} - \frac{\nu_{23}}{2a_{22}} + \frac{a_{21}\nu_{23}}{2a_{11}a_{22}} & -\delta_2 - \frac{\nu_{12}}{2a_{11}} - \frac{a_{21}\nu_{23}}{2a_{11}a_{22}} & \delta_2 + \delta_3  + \frac{\nu_{23}}{a_{22}} &  -\delta_3 - \frac{\nu_{23}}{2a_{22}} \\
    \frac{\nu_{13}}{2a_{11}} + \frac{\nu_{23}}{2a_{22}} - \frac{a_{21}\nu_{23}}{2a_{11}a_{22}} & - \frac{\nu_{13}}{2a_{11}} + \frac{a_{21}\nu_{23}}{2a_{11}a_{22}}  & -\delta_3 - \frac{\nu_{23}}{2a_{22}} & \delta_3
  \end{bmatrix} .
\]
With this at hand, we define the quadratic form $\mclQ_s : H^{s+1} \to \Real$ as
\begin{equation}
\label{eq:QQuadFormUstab}
  \mclQ_s(w_1, \ldots, w_{s+1}) = \sum_{i,j=1}^{s+1} q(s)_{i,j} (w_i,w_j).
\end{equation}

The relevant stability notion is given next.

\begin{definition}[$U$-stability]
\label{def:Ustability}
  We say that the DIRK scheme \eqref{eq:Stages}--\eqref{eq:Advance} is $U$-stable if $b_i>0$ for all $i=1,\ldots, s$ and the matrix $\msfQ(s)$, or equivalently the bilinear form $\mclQ_s$ of \eqref{eq:QQuadFormUstab}, is positive semi-definite.
\end{definition}

The main use of this stability notion is summarized in the following result.

\begin{proposition}[stability]
\label{prop:BochnerStable}
  Assume that \AJS{the} DIRK scheme \eqref{eq:Stages}--\eqref{eq:Advance} is $U$-stable. Then there are positive $\{\nu_i\}_{i=1}^s$ such that, for every \AJS{$\mclN \in \Natural$, every time partition $\{t_n\}_{n=0}^\mclN$, and every $n \in \{0, \ldots, \mclN-1\}$ we have the local energy identity
  \begin{equation}
    \label{eq:BochnerStableLocal}
    \frac12 \nrm{u_{n+1}}_H^2 + \mclQ_s(u_n,v_{n,1},\ldots, v_{n,s}) + \tau_n \sum_{i=1}^s \nu_i \langle \mclA v_{n,i} , v_{n,i} \rangle =
    \frac12 \nrm{u_n}_H^2 + \tau_n \sum_{i=1}^s \nu_i \langle f_{n,i}, v_{n,i} \rangle ,
  \end{equation}
  which in turn implies the global estimate,
  \begin{equation}
  \label{eq:BochnerStableGlobal}
    \frac12\max_{n=0}^{\mclN} \nrm{u_{n}}_H^2 + \sum_{n=1}^{\mclN} \tau_n \sum_{i=1}^s \nu_i \nrm{ v_{n,i} }_V^p \lesssim 
    \frac12 \nrm{u_0}_H^2 + \sum_{n=1}^{\mclN} \tau_n \sum_{i=1}^s \nu_i \nrm{ f_{n,i} }_{V'}^{p'},
  \end{equation}
  and the dual norm bound}
  \begin{equation}
  \label{eq:DualNormBound}
    \sum_{n=0}^{\mclN-1} \tau_n \nrm{ \frac{u_{n+1} - u_n}{\tau_n} }_{V'}^{q'} \lesssim \nrm{u_0}_H^2 + \sum_{n=1}^{\mclN} \tau_n \sum_{i=1}^s \nu_i \nrm{f_{n,i}}_{V'}^{p'}.
  \end{equation}
\end{proposition}
\begin{proof}
  Corollaries 4.2 and 5.2, and Propositions 6.3 and 6.4 in \cite{MR4673496}.
\end{proof}

We present a slight improvement on this result by explicitly computing the value of the weights $\{\nu_i\}_{i=1}^s$.

\begin{lemma}[quadrature weights]
  Assume that DIRK scheme \eqref{eq:Stages}--\eqref{eq:Advance} is $U$-stable. Then, for every $i = 1,\ldots, s$, we have $\nu_i = b_i$.
\end{lemma}
\begin{proof}
  For $s=2$, we first observe that
  \[
    \nu_2 = \nu_{22} = a_{22}\lambda_2.
  \]
  Since $\msfA^\top \bflambda = \bfb$ and $\msfA$ is lower triangular, the second row of this system dictates $a_{22}\lambda_2 = b_2$, thus $\nu_2 = b_2$. By the consistency requirement $\nu_1 + \nu_2 = b_1 + b_2$, it follows immediately that $\nu_1 = b_1$.

  In the three stage case, that is $s=3$,  we begin by noting that $\nu_3 = \nu_{33} = a_{33}\lambda_3 = b_3$. Next, using that $\nu_2 = \nu_{22} + \nu_{23}$ and expanding we obtain
  \begin{align*}
    \nu_2 &= a_{22}[\lambda_2(1-\lambda_3) + (2-\lambda_3)\lambda_3] + \lambda_3[a_{32} + a_{22}(-2+\lambda_2+\lambda_3)]
    \\
    &= a_{22}\lambda_2 - a_{22}\lambda_2\lambda_3 + 2a_{22}\lambda_3 - a_{22}\lambda_3^2 + a_{32}\lambda_3 - 2a_{22}\lambda_3 + a_{22}\lambda_2\lambda_3 + a_{22}\lambda_3^2
    \\
    &= a_{22}\lambda_2 + a_{32}\lambda_3,
  \end{align*}
  which is the second row of $\msfA^\top \bflambda = \bfb$, hence $\nu_2 = b_2$. Finally, the global consistency constraint $\sum_{i=1}^3 \nu_i = \sum_{i=1}^3 b_i$ enforces $\nu_1 = b_1$.
\end{proof}

\subsubsection{$B$-stability}

We now briefly discuss the notion of $B$-stability, or algebraic stability. While this can be applied to any RK scheme, we restrict our attention to the case of DIRK schemes as in \eqref{eq:ButcherTable}. For that, we introduce the so-called M-matrix of the scheme
\begin{equation}
\label{eq:MMatrixBstab}
  \msfM \coloneqq \msfB \msfA + \msfA^\top \msfB - \bfb\bfb^\top, \qquad \qquad \msfB \coloneqq \diag(b_1, \ldots, b_s).
\end{equation}

\begin{definition}[$B$-stability]
\label{def:Bstability}
  We say that the DIRK scheme \eqref{eq:Stages}--\eqref{eq:Advance} is $B$-stable if $b_i>0$ for all $i=1,\ldots, s$ and the matrix $\msfM$ is positive semi-definite.
\end{definition}


The terminology used in Definition~\ref{def:Bstability} goes back to the classical theory of nonlinear stability for Runge-Kutta methods. Notice that linear stability notions, such as $A$-stability, are stated in terms of a scalar linear equation and therefore are not readily applicable to nonlinear problems. In short, $A$-stability does not imply that the scheme inherits the contractive or dissipative properties of the continuous problem. On the other hand, $B$-stability is suited for nonlinear dissipative systems. In the original finite-dimensional form, one considers two solutions of $\bfy' = \AJS{\bff}(t,\bfy)$ under a one-sided Lipschitz, or dissipativity, condition of the form
\begin{equation}
\label{eq:VecFieldMonotone}
  (\AJS{\bff}(t,\bfy)-\AJS{\bff}(t,\bfz))^\top(\bfy-\bfz) \leq 0.
\end{equation}
This condition guarantees that the distance between any two flows is nonincreasing in time. We say that a Runge-Kutta method is $B$-stable if its discrete solution operator preserves this property. This notion was introduced in \cite{butcher1975stability} and then characterized in \cite{MR525638,MR518683} and extended beyond Runge-Kutta methods in \cite{burrage1980non}; see also \cite[Section IV.12]{MR2657217} for a comprehensive classical treatment. In this terminology, algebraic stability is the coefficient-level condition presented in Definition ~\ref{def:Bstability}. Under the usual assumptions, algebraic stability is the easily verifiable condition behind $B$-stability; in the present work, where all weights are assumed positive and the coefficient matrix $\msfA$ is fixed, we use the coefficient condition in Definition~\ref{def:Bstability} as our definition.

What makes $B$-stability a desirable property is that it turns condition \eqref{eq:VecFieldMonotone} on the vector field \AJS{$\bff$} into a discrete contraction estimate. The quadratic form corresponding to $\msfM$ is exactly what appears when one compares the difference of two Runge-Kutta steps and expands the squared norm. Hence, the \AJS{nonnegativity} of $\msfM$ serves the same purpose in the discrete case as monotonicity of the operator in the continuous case. This is the fundamental reason that algebraically stable methods are suitable to dissipative equations, monotone operator equations, and gradient problems.

In addition, algebraic stability is relevant for preserving a wide amount of properties in dissipative dynamical systems. In \cite{MR1293524} it is shown that algebraically stable Runge-Kutta methods have an absorbing set and remain dissipative regardless of the chosen time step. It was proven in \cite{MR3194795} that algebraically stable methods monotonically decrease the discrete energy at each step under step-size restrictions. The paper \cite{MR4121887} studies the dissipativity of such schemes applied to convex gradient problems. More recent applications include \cite{MR4007568}, where positive definiteness conditions of stiffly accurate Runge-Kutta methods are used to categorize energy-stable solvers for convex gradient flows.

These ideas have also been considered for infinite-dimensional problems, see \cite{MR2329662} where error estimates are derived for algebraically stable Runge-Kutta methods solving nonlinear dissipative evolution equations in Hilbert spaces, and \cite{MR2566077} shows that algebraically stable DIRK schemes converge unconditionally when applied to dissipative evolution equations. In the Gelfand-triple setting, \cite{MR2600543} provides proof of convergence of stiffly accurate implicit Runge-Kutta schemes for nonlinear evolution problems with monotone operators, and the stability assumptions imply algebraic stability.

\AJS{Below, in Section~\ref{sec:GradFlows} we provide some positive and negative results regarding the energy dissipation of $B$-stable schemes when applied to gradient flows of the form \eqref{eq:EVI}.  }

\section{Equivalence between $U$- and $B$-stability}
\label{sec:equivalence}

This section presents the first main result of our work. We establish the equivalence  between the two notions of stability we are interested in, namely, $U$- and $B$-stability for $2$-stage and $3$-stage DIRK schemes. We achieve this by showing that the matrix $\msfM$, introduced in \eqref{eq:MMatrixBstab}, is positive semi-definite if and only if the quadratic form $\mclQ_s$, defined in \eqref{eq:QQuadFormUstab}, has the same property.

To achieve this, we begin by recalling that similarity transformations do not alter the spectrum. Thus, since $\msfA$ is invertible, we define
\begin{equation}
\label{eq:Mtilde}
  \tilde{\msfM} \coloneqq \msfA^{-\top}\msfM\msfA^{-1} = \msfA^{-\top}\msfB + \msfB\msfA^{-1} - \bflambda\bflambda^\top,
\end{equation}
where we recall that $\bflambda = \msfA^{-\top} \bfb$. Our strategy will then be to show that $\tilde{\msfM}$ is similar to a matrix related to the quadratic form $\mclQ_s$. We do this separately for $s = 2$ and $s=3$.

\subsection{The two stage case}

We begin with the case $s=2$.

\begin{theorem}[equivalence for $s=2$]
  For any two stage DIRK scheme with an invertible matrix $\msfA$ and $b_i >0$, $i=1,2$ $U$-stability is equivalent to $B$-stability.
\end{theorem}
\begin{proof}
  Since both notions require that $b_i >0$ for $i=1,2$ we only need to consider the matrix $\tilde{\msfM}$ and the quadratic form $\mclQ_2$.

  Next we observe that
  \[
    \mclQ_2(w_1,w_2,w_3) = \mclK_2(w_2-w_1,w_3-w_2),
  \]
  where the quadratic form $\mclK_2 : H^2 \to \Real$ is defined as
  \[
    \mclK_2(\zeta_1,\zeta_2) \coloneqq \delta_1 \nrm{ \zeta_1}_H^2 + \delta_2 \nrm{ \zeta_2 }_H^2 - \frac{\nu_{12}}{a_{11}} \left( \zeta_1, \zeta_2 \right).
  \]
  We collect the coefficients in the symmetric matrix $\msfK \in \Real^{2 \times 2}$:
  \begin{equation*}
    \msfK =
    \begin{bmatrix}
    \delta_1 & -\frac{\nu_{12}}{2a_{11}} \\
    -\frac{\nu_{12}}{2a_{11}} & \delta_2
    \end{bmatrix}
  \end{equation*}
  and observe that $\mclQ_2$ is positive semi-definite if and only if the matrix $\msfK$ has the same property. The advantage of this change of variables now is that both $\msfK$ and $\tilde{\msfM}$ have the same dimensions, so we can make assertions about their spectrum.

  To link $\msfK$ and $\tilde{\msfM}$, we explicitly compute the entries of $\tilde{\msfM}$. It is immediate that
  \begin{equation*}
    \msfA^{-\top}\msfB =
    \begin{bmatrix}
      \frac{b_1}{a_{11}} & -\frac{a_{21}b_2}{a_{11}a_{22}} \\
      0 & \frac{b_2}{a_{22}}
    \end{bmatrix},
    \qquad \qquad
    \msfB\msfA^{-1} =
    \begin{bmatrix}
      \frac{b_1}{a_{11}} & 0 \\
      -\frac{a_{21}b_2}{a_{11}a_{22}} & \frac{b_2}{a_{22}}
    \end{bmatrix},
  \end{equation*}
  and
  \[
    \lambda_1 = \frac{b_1}{a_{11}} - \frac{a_{21}b_2}{a_{11}a_{22}}, \qquad \qquad \lambda_2 = \frac{b_2}{a_{22}}.
  \]

  Define $\bfone = [1, 1]^\top$ and
  \begin{equation*}
    \msfL \coloneqq
    \begin{bmatrix}
      1 & 0 \\
      1 & 1
    \end{bmatrix}
    =
    \left[ \bfone \ \bfe_2 \right] \in \Real^{2 \times 2},
  \end{equation*}
  where, as usual, we set $\bfe_2 = [0, 1]^\top$. Our next step is to compute $\frac{1}{2} \msfL^\top \tilde{\msfM} \msfL$. Since
  \begin{equation*}
    \frac{1}{2} \msfL^\top \tilde{\msfM} \msfL = \frac{1}{2}
    \begin{bmatrix}
      \bfone^\top \tilde{\msfM} \bfone & \bfone^\top \tilde{\msfM} \bfe_2 \\
      \bfe_2^\top \tilde{\msfM} \bfone & \bfe_2^\top \tilde{\msfM} \bfe_2
    \end{bmatrix}
  \end{equation*}
  we may use the explicit values of the entries of $\tilde{\msfM}$ derived above to obtain
  \[
    \frac{1}{2} \bfone^\top \tilde{\msfM} \bfone = \frac{1}{2} \left( \bfone^\top \msfA^{-\top}\msfB \bfone + \bfone^\top \msfB\msfA^{-1} \bfone - (\bfone^\top \bflambda)^2 \right) 
    = \frac{1}{2} \left( 2(\lambda_1 + \lambda_2) - (\lambda_1 + \lambda_2)^2 \right) = \delta_1.
  \]
  Similarly,
  \[
    \frac{1}{2} \bfe_2^\top \tilde{\msfM} \bfe_2 = \frac{1}{2} \left( \bfe_2^\top \msfA^{-\top}\msfB \bfe_2 + \bfe_2^\top \msfB \msfA^{-1} \bfe_2 - (\bfe_2^\top \bflambda)^2 \right) 
    = \frac{1}{2} \left( 2 \frac{b_2}{a_{22}} - \lambda_2^2 \right) = \frac{1}{2} \lambda_2(2 - \lambda_2) = \delta_2.
  \]
  For the off-diagonal term we have
  \begin{align*}
    \frac{1}{2} \bfone^\top \tilde{\msfM} \bfe_2 &= \frac{1}{2} \left( \bfone^\top \msfA^{-\top} \msfB \bfe_2 + \bfone^\top \msfB \msfA^{-1} \bfe_2 - (\bfone^\top \bflambda)(\bflambda^\top \bfe_2) \right).
  \end{align*}
  Next, using that
  \[
    \msfA^{-\top}\msfB \bfe_2 = \begin{bmatrix} -\frac{a_{21}b_2}{a_{11}a_{22}} \\ \frac{b_2}{a_{22}} \end{bmatrix},
    \qquad
    \msfB \msfA^{-1} \bfe_2 = \begin{bmatrix} 0 \\ \frac{b_2}{a_{22}} \end{bmatrix},
  \]
  we see that
  \[
    \frac{1}{2} \bfone^\top \tilde{\msfM} \bfe_2 = \lambda_2 \left( 1 - \frac{ a_{21} }{ a_{11} } \right).
  \]
  Finally, using the computed value of $\lambda_2$,
  \begin{align*}
    \frac{1}{2} \bfone^\top \tilde{\msfM} \bfe_2 = \frac{1}{2} \left[ -\lambda_2 \left( \frac{a_{21}}{a_{11}} + \lambda_1 + \lambda_2 - 2 \right) \right] = -\frac{\nu_{12}}{2a_{11}}.
  \end{align*}
  This proves that $\msfK = \frac{1}{2} \msfL^\top \tilde{\msfM} \msfL$ and so the matrices $\msfK$ and $\frac12\msfM$ have the same spectrum. In other words, $\msfK$ is positive semi-definite if and only if $\msfM$ is positive semi-definite.
\end{proof}

\subsection{The three stage case}

The result, and method of proof, for $s=3$ is the same. It is only more involved as the expressions are larger. As in \cite{MR4673496}, a computer algebra system may be useful in this situation.

\begin{theorem}[equivalence \AJS{for $s=3$}]
  For any $3$-stage DIRK scheme with an invertible matrix $\msfA$, $U$-stability is equivalent to $B$-stability.
\end{theorem}
\begin{proof}
  \AJS{As in the two-stage case, both notions require that, for $i=1,2,3$, $b_i>0$; so we focus on the spectra of matrices $\tilde{\msfM}$ and $\msfK$, which is to be defined}.
  
  We rewrite the quadratic form $\mclQ_3$ as
  \[
    \mclQ_3(w_1,w_2,w_3,w_4) = \mclK_3(w_2-w_1,w_3 - w_2, w_4 - w_3),
  \]
  where the coefficient matrix $\msfK$ has the following entries:
  \begin{align*}
    k_{11} &= \delta_1, \quad k_{22} = \delta_2, \quad k_{33} = \delta_3, \\
    k_{12} &= k_{21} = \frac{1}{2} \left( -\frac{\nu_{12}}{a_{11}} - \frac{\nu_{13}}{a_{11}} \right), \\
    k_{13} &= k_{31} = \frac{1}{2} \left( -\frac{\nu_{13}}{a_{11}} - \frac{\nu_{23}}{a_{22}} + \frac{a_{21}\nu_{23}}{a_{11}a_{22}} \right), \\
    k_{23} &= k_{32} = \frac{1}{2} \left( -\frac{\nu_{23}}{a_{22}} \right).
  \end{align*}

  Let us now compute the entries of $\tilde{\msfM}$. Direct computations yield:
  \begin{equation*}
    \msfA^{-\top}\msfB =
    \begin{bmatrix}
      \frac{b_1}{a_{11}} & -\frac{a_{21}b_2}{a_{11}a_{22}} & \frac{(-a_{22}a_{31}+a_{21}a_{32})b_3}{a_{11}a_{22}a_{33}} \\
      0 & \frac{b_2}{a_{22}} & -\frac{a_{32}b_3}{a_{22}a_{33}} \\
      0 & 0 & \frac{b_3}{a_{33}}
    \end{bmatrix},
    \qquad\qquad
    \msfB\msfA^{-1} =
    \begin{bmatrix}
      \frac{b_1}{a_{11}} & 0 & 0 \\
      -\frac{a_{21}b_2}{a_{11}a_{22}} & \frac{b_2}{a_{22}} & 0 \\
      \frac{(-a_{22}a_{31}+a_{21}a_{32})b_3}{a_{11}a_{22}a_{33}} & -\frac{a_{32}b_3}{a_{22}a_{33}} & \frac{b_3}{a_{33}}
    \end{bmatrix}.
  \end{equation*}
  Notice that the upper left $2 \times 2$ block is the same as in the two stage case. The row sums of $\msfA^{-\top}\msfB$ explicitly define the elements of the extrapolation vector $\bflambda$.

  We define the change-of-basis matrix
  \begin{equation*}
    \msfL \coloneqq
    \begin{bmatrix}
      1 & 0 & 0 \\
      1 & 1 & 0 \\
      1 & 1 & 1
    \end{bmatrix}
    = \begin{bmatrix}
        \bfl_1 & \bfl_2 & \bfl_3
      \end{bmatrix} \in \Real^{3 \times 3}.
  \end{equation*}
  We now compute $\frac{1}{2} \bfl_i^\top \tilde{\msfM} \bfl_i$:
  \begin{align*}
    \frac{1}{2} \bfl_3^\top \tilde{\msfM} \bfl_3 &= \frac{1}{2} \left( 2\left(\frac{b_3}{a_{33}}\right) - \lambda_3^2 \right) = \frac{1}{2}(2\lambda_3 - \lambda_3^2) = \delta_3 = k_{33}. \\
    \frac{1}{2} \bfl_2^\top \tilde{\msfM} \bfl_2 &= \frac{1}{2} \left( 2\left( \frac{b_2}{a_{22}} - \frac{a_{32}b_3}{a_{22}a_{33}} + \frac{b_3}{a_{33}} \right) - (\lambda_2+\lambda_3)^2 \right) = \frac{1}{2} \left( 2(\lambda_2+\lambda_3) - (\lambda_2+\lambda_3)^2 \right) = \delta_2 = k_{22}. \\
    \frac{1}{2} \bfl_1^\top \tilde{\msfM} \bfl_1 &= \frac{1}{2} \left( 2(\lambda_1+\lambda_2+\lambda_3) - (\lambda_1+\lambda_2+\lambda_3)^2 \right) = \delta_1 = k_{11},
  \end{align*}
  and we see that the diagonal entries of $\frac12\msfL^\top \tilde{\msfM} \msfL$ coincide with those of $\msfK$.
  
  We now evaluate the off-diagonal entries of $\frac12\msfL^\top \tilde{\msfM} \msfL$. To compute $\frac{1}{2} \bfl_2^\top \tilde{\msfM} \bfl_3$ we observe that the operation $\bfl_2^\top \msfA^{-\top} \msfB \bfl_3$ extracts the sum of the second and third rows of the third column of $\msfA^{-\top}\msfB$, which is 
  \[
    \frac{b_3}{a_{33}} - \frac{a_{32}b_3}{a_{22}a_{33}} = \lambda_3\left(1 - \frac{a_{32}}{a_{22}}\right).
  \]
  The transpose block gives $\bfl_2^\top \msfB \msfA^{-1} \bfl_3 = \frac{b_3}{a_{33}} = \lambda_3$. Gathering,
  \[
    \frac{1}{2} \bfl_2^\top \tilde{\msfM} \bfl_3 = \frac{1}{2} \left[ \lambda_3 \left(1 - \frac{a_{32}}{a_{22}}\right) + \lambda_3 - (\lambda_2 + \lambda_3)\lambda_3 \right] 
    = -\frac{1}{2a_{22}} \lambda_3 \left[ a_{32} + a_{22}(-2 + \lambda_2 + \lambda_3) \right] = -\frac{\nu_{23}}{2a_{22}} = k_{23}.
  \]

  We now compute $\frac{1}{2} \bfl_1^\top \tilde{\msfM} \bfl_2$. We have 
  \[
    \bfl_1^\top \msfA^{-\top} \msfB \bfl_2 = \lambda_2\left(1 - \frac{a_{21}}{a_{11}} \right) + \lambda_3\left(1 - \frac{a_{31}}{a_{11}} \right).
  \]
  The transpose block gives $\bfl_1^\top \msfB \msfA^{-1} \bfl_2 = \lambda_2 + \lambda_3$. We also have that
  \begin{align*}
    2 k_{12} &= \lambda_2\left(1 - \frac{a_{21}}{a_{11}}\right) + \lambda_3\left(1 - \frac{a_{31}}{a_{11}}\right) + (\lambda_2 + \lambda_3) - (\lambda_1+\lambda_2+\lambda_3)(\lambda_2+\lambda_3) \\
    &= -\frac{a_{21}}{a_{11}}\lambda_2 - \frac{a_{31}}{a_{11}}\lambda_3 + (\lambda_2+\lambda_3)\left[ 2 - \lambda_1 - \lambda_2 - \lambda_3 \right].
  \end{align*}
  To match the Bochner stability formulation, we expand $-\frac{\nu_{12}}{a_{11}} - \frac{\nu_{13}}{a_{11}}$ directly from their definitions:
  \begin{align*}
    -\frac{\nu_{12}}{a_{11}} - \frac{\nu_{13}}{a_{11}} &= -\frac{a_{21}}{a_{11}} \left[ \lambda_2(1-\lambda_3) + (2-\lambda_3)\lambda_3 + \lambda_3(-2+\lambda_2+\lambda_3) \right] \\
    &\quad - \left[ \lambda_2(-2+\lambda_1+\lambda_2) + 2(-1+\lambda_2)\lambda_3 + \lambda_3^2 \right] - \frac{a_{31}}{a_{11}}\lambda_3 - \lambda_1\lambda_3.
  \end{align*}
  The bracketed term factored by $-\frac{a_{21}}{a_{11}}$ simplifies to $\lambda_2$. The remaining terms evaluate to $(\lambda_2+\lambda_3)[2 - \lambda_1 - \lambda_2 - \lambda_3] - \frac{a_{31}}{a_{11}}\lambda_3$, thus 
  \[
    \frac{1}{2} \bfl_1^\top \tilde{\msfM} \bfl_2 = \frac{1}{2} \left( -\frac{\nu_{12}}{a_{11}} - \frac{\nu_{13}}{a_{11}} \right) = k_{12}.
  \]

  Finally, we find $\frac{1}{2} \bfl_1^\top \tilde{\msfM} \bfl_3$. The term $\bfl_1^\top \msfA^{-\top} \msfB \bfl_3$ sums the entire third column of $\msfA^{-\top} \msfB$, yielding 
  \[
    \bfl_1^\top \msfA^{-\top} \msfB \bfl_3 = \lambda_3 \left[ 1 - \frac{a_{32}}{a_{22}} - \frac{a_{31}}{a_{11}} + \frac{a_{21}a_{32}}{a_{11}a_{22}} \right].
  \]
  The transpose block gives $\lambda_3$. Therefore,
  \[
    2 k_{13} = \lambda_3 \left[ 2 - \frac{a_{31}}{a_{11}} - \frac{a_{32}}{a_{22}} + \frac{a_{21}a_{32}}{a_{11}a_{22}} \right] - (\lambda_1+\lambda_2+\lambda_3)\lambda_3 
    = \lambda_3 \left[ 2 - \lambda_1 - \lambda_2 - \lambda_3 - \frac{a_{31}}{a_{11}} - \frac{a_{32}}{a_{22}} + \frac{a_{21}a_{32}}{a_{11}a_{22}} \right].
  \]
  We expand $-\frac{\nu_{13}}{a_{11}} - \frac{\nu_{23}}{a_{22}} + \frac{a_{21}\nu_{23}}{a_{11}a_{22}}$ from their definitions to match:
  \begin{align*}
    -\frac{\nu_{13}}{a_{11}} - \frac{\nu_{23}}{a_{22}} + \frac{a_{21}\nu_{23}}{a_{11}a_{22}} &= \lambda_3 \left[ -\frac{a_{31}}{a_{11}} - \lambda_1 - \frac{a_{21}}{a_{11}}(-2+\lambda_2+\lambda_3) \right] \\
    &\quad + \lambda_3 \left[ -\frac{a_{32}}{a_{22}} + 2 - \lambda_2 - \lambda_3 \right] + \lambda_3 \left[ \frac{a_{21}}{a_{11}}\left( \frac{a_{32}}{a_{22}} - 2 + \lambda_2 + \lambda_3 \right) \right].
  \end{align*}
  Grouping the terms factored by $\frac{a_{21}}{a_{11}}$ gives us $\frac{a_{21}a_{32}}{a_{11}a_{22}}$. The remaining terms match the expansion above, hence $\frac{1}{2} \bfl_1^\top \tilde{\msfM} \bfl_3 = k_{13}$.
  
  We have proved that $\msfK = \frac{1}{2} \msfL^\top \tilde{\msfM} \msfL$, and this shows that $\msfK$ and $\frac12\msfM$ have the same spectrum. One of these matrices is positive semi-definite, if and only if the other one is.
\end{proof}

\section{Convergence of $U$-stable DIRK schemes for linear evolution problems}
\label{sec:ConvLinear}

As a first application of the equivalence we presented here we prove the convergence under minimal regularity assumptions for problems with a linear coercive operator. Thus we further assume that the operator $\mclA$ is linear, and we set $p=q=2$. Notice that, in this case the function \AJS{$\mclM$} is actually a constant, and the local solvability assumption holds for every $\gamma>0$.

The main idea will be to construct, from the sequence of stages $\{v_{n,i}\}_{i=1,n=1}^{s,\mclN}$ and updates $\{u_n\}_{n=0}^\mclN$, several families of functions that will converge, when $\bftau \to \bfzero$, to the solution of \eqref{eq:VarSolution}. 

\subsection{Time discrete functions}

Let us now describe how, out of sequences, we may construct time dependent functions. In what follows $X$ is a Banach space.

Given $\{w_n\}_{n=0}^\mclN \subset X$ we define its piecewise constant approximation $\bar{w}_\bftau \in L^\infty(0,T;X)$ as
\[
  \bar{w}_\bftau(t) \coloneqq w_{n+1}, \qquad t \in (t_n, t_{n+1}], \qquad n \in \{0, \ldots, \mclN-1\}.
\]
The piecewise linear interpolant $\hat{w}_\bftau \in C^{0,1}([0,T];X)$ is
\[
  \hat{w}_\bftau(t) \coloneqq w_n + \frac{ w_{n+1} - w_n }{\tau_n} \left( t -t_n \right), \qquad t \in I_n, \qquad n \in \{0, \ldots, \mclN-1\}.
\]
Note that
\[
  \diff_t \hat{w}_\bftau(t) = \frac{ w_{n+1} - w_n }{\tau_n}, \qquad t \in \mathring{I}_n, \qquad n \in \{0, \ldots, \mclN-1\}.
\]

Next, if $\{w_{n,i} \}_{i=1,n=0}^{s,\mclN-1} \subset X$, we define the function $\tilde{w}_\bftau \in L^\infty(0,T;X)$, such that $\tilde{w}_{\bftau|\mathring{I}_n} \in \mbbP_{s-1} \otimes X$ as
\[
  \tilde{w}_\bftau = \sum_{i=1}^s w_{n,i} \ell_{n,i}(t), \qquad t \in \mathring{I}_n, \qquad n \in \{0, \ldots, \mclN-1\},
\]
where the cardinal basis functions $\{\ell_{n,i} \}_{i=1}^s \subset \mbbP_{s-1}$ are defined as
\[
  \ell_{n,i}(t) = \prod_{j=1,j\neq i}^s \frac{ t-(t_n + c_j \tau_n) }{\tau_n(c_i-c_j)}.
\]

With this notation, we study the effect of some quadrature rules.

\begin{lemma}[quadrature]
\label{lem:quadrature}
  Let $s \in \{2, 3\}$, $\mclN \in \Natural$, and assume that the $s$-stage DIRK scheme is consistent to order $p=2(s-1) \leq 5$. Then, for any partition and all $n \in \{0, \ldots, \mclN-1\}$, we have
  \[
    \int_{t_n}^{t_{n+1}} q(t) \diff t = \tau_n \sum_{i=1}^s b_i q(t_n + c_i \tau_n), \qquad \forall q \in \mbbP_{2s-2}.
  \]
  As a consequence, for any $\{w_{n,i} \}_{i=1,n=0}^{s,\mclN-1} \subset X$ and all $n \in \{0, \ldots, \mclN-1\}$, we have
  \begin{equation}
  \label{eq:Integration}
    \int_{t_n}^{t_{n+1}} \tilde{w}_\bftau(t) \diff t = \tau_n \sum_{i=1}^s b_i w_{n,i}.
  \end{equation}
\end{lemma}
\begin{proof}
  Define the quadrature rule
  \[
    Q[q] \coloneqq \tau_n \sum_{i=1}^s b_i q(t_n + c_i \tau_n).
  \]
  
  Let $p=1$, then we have
  \[
    Q[1] = \tau_n \sum_{i=1}^s b_i = \tau_n \bfb^\top \bfone = \tau_n = \int_{t_n}^{t_{n+1}} 1 \diff t.
  \]
  
  Let now $p=2$,
  \[
    Q[t] = \tau_n \sum_{i=1}^s b_i (t_n + c_i \tau_n) = \tau_n \left( t_n + \tau_n \bfb^\top \bfc \right) = \tau_n \left( t_n + \frac{\tau_n}2 \right) = \int_{t_n}^{t_{n+1}} t \diff t,
  \]
  where we used the second order condition $\bfb^\top \bfc = \tfrac12$.
  
  If $p = 3$ we must have $\bfb^\top \bfc^2 =\tfrac13 $. Then,
  \begin{align*}
    Q[t^2] &= \tau_n \sum_{i=1}^s b_i (t_n + c_i \tau_n)^2 = \tau_n \sum_{i=1}^s b_i \left( t_n^2 + 2t_n c_i \tau_n + c_i^2 \tau_n^2  \right) 
    \\
    &= \tau_n \left( t_n^2  + t_n \tau_n + \tau_n^2 \sum_{i=1}^s b_i c_i^2 \right) = \int_{t_n}^{t_{n+1}} t^2 \diff t,
  \end{align*}
  where we, again, used the order conditions.

  Next, for $p =4$, one has $\bfb^\top \bfc^3 = \frac{1}{4}$, hence
   \begin{align*}
    Q[t^3] & = \tau_n \sum_{i=1}^s b_i (t_n + c_i \tau_n)^3 = \tau_n \sum_{i=1}^s b_i \left( t_n^3 + 3 t_n^2 c_i \tau_n + 3 t_n c_i^2 \tau_n^2 + c_i^3 \tau_n^3\right) \\
    & = \tau_n \left(t_n^3 + \frac{3}{2}t_n^2 \tau_n +t_n \tau_n^2 +\tau_n^3 \sum_{i=1}^s b_i c_i^3 \right) = \int_{t_n}^{t_{n+1}} t^3 \diff t
  \end{align*}
  where the order conditions were once again applied.
  
  For $p=5$, the order conditions give $\bfb^\top \bfc^4 = \frac{1}{5}$, thus
     \begin{align*}
    Q[t^4] & = \tau_n \sum_{i=1}^s b_i (t_n + c_i \tau_n)^4 = \tau_n \sum_{i=1}^s b_i \left( t_n^4 + 4 t_n^3 c_i \tau_n + 6 t_n^2 c_i^2 \tau_n^2 + 4 t_n c_i^3\tau_n^3 + c_i^4 \tau_n^4\right) \\
    & = \tau_n \left(t_n^4 + 2t_n^3 \tau_n + 2 t_n^2 \tau_n^2 + t_n \tau_n^3 +\tau_n^4 \sum_{i=1}^s b_i c_i^4 \right) = \int_{t_n}^{t_{n+1}} t^4 \diff t.
  \end{align*}
  
  Finally, if $s \in \{2,3\}$, we know that $s$-stage DIRK schemes are consistent to at least order $s$. It suffices then to observe that $\tilde{w}_{\bftau|\mathring{I}_n} \in \mbbP_{s-1}$.
\end{proof}

\begin{remark}[\AJS{order conditions}]
  In the proof of Lemma~\ref{lem:quadrature} we used the order conditions $\bfb^\top \bfc^k = \frac{1}{k+1}$, $k \leq p$. These come from applying the RK scheme to
  \[
    y' = y, \qquad y(0) = 1,
  \]
  and comparing the solution of the scheme with the coefficients of the Taylor expansions of the exact solution. One can show that the third order condition for the scheme $\bfb^\top \msfA \bfc = \frac{1}{6}$ is equivalent to the third order condition above, namely, $\bfb^\top \bfc^2 = \tfrac13$. Similar arguments show the equivalence of the fourth order conditions $\bfb^\top \msfA^2 \bfc = \frac{1}{24}$ and $\bfb^\top \bfc^3 =\frac{1}{4}$ and fifth order conditions $\bfb^\top \msfA^3 \bfc = \frac{1}{120}$ and $\bfb^\top \bfc^4 = \frac{1}{5}$.
\end{remark}

Let us now multiply \eqref{eq:Advance} by $\tau_n$ and use Lemma~\ref{lem:quadrature} to obtain, equivalently, that, for all $w \in V$, we have
\begin{equation}
\label{eq:AdvaceIntLoc}
  \int_{t_n}^{t_{n+1}} \left( \diff_t \hat{u}_\bftau(t), w \right) \diff t + \int_{t_n}^{t_{n+1}} \langle \mclA \tilde{v}_\bftau(t), w \rangle \diff t = \int_{t_n}^{t_{n+1}} \langle \tilde{f}_\bftau(t), w \rangle \diff t.
\end{equation}

So far, we have not specified how the sequence $\{f_{n,i} \}_{i=1,n=0}^{s,\mclN-1}$ is constructed. We do that now and draw conclusions from this choice.

\begin{proposition}[approximation of $f$]
\label{prop:FApprox}
  Let $f_\varepsilon \in C([0,T];V')$ be such that, as $\varepsilon \to 0$, $f_\varepsilon \to f$ in $L^2(0,T;V')$. Define, for $n \in \{0, \ldots, \mclN-1\}$ and $i \in \{1,\ldots, s\}$,
  \[
    f_{n,i} = f_\varepsilon(t_n + c_i \tau_n).
  \]
  Let $\varepsilon = \varepsilon(\bftau)$ be such that $\varepsilon \to 0$ as $\bftau \to \bfzero$. Then, as $\bftau \to \bfzero$, $\tilde{f}_\bftau \to f$ in $L^2(0,T;V')$.
\end{proposition}
\begin{proof}
  We compute
  \[
    \| f - \tilde{f}_\bftau \|_{L^2(0,T;V')}^2 \leq 2\| f - f_\varepsilon \|_{L^2(0,T;V')}^2  + 2  \sum_{n = 0}^{\mclN-1} \int_{I_n} \| f_\varepsilon(t) - \tilde{f}_\bftau(t) \|_{V'}^2 \diff t,
  \]
  and focus on the second term. For $t \in I_n$, we have that, owing to the definition of $\{f_{n,i}\}_{i=1,n=0}^{s,\mclN-1}$,
  \[
    \tilde{f}_\bftau(t) - f_\varepsilon(t) = \sum_{i=1}^s \left[ f_\varepsilon(t_n + c_i \tau_n ) - f_\varepsilon(t) \right] \ell_{n,i}(t),
  \]
  where we used that $\sum_{i=1}^s \ell_{n,i}(t) = 1$. This gives,
  \[
    \int_{I_n} \| f_\varepsilon(t) - \tilde{f}_\bftau(t) \|_{V'}^2 \diff t \lesssim \sum_{i=1}^s \int_{I_n} |\ell_{n,i}(t)|^2 \| f_\varepsilon(t_n + c_i \tau_n ) - f_\varepsilon(t) \|_{V'}^2 \diff t.
  \]
  Since $f_\varepsilon \in C([0,T];V')$, it is uniformly continuous on $[0,T]$. In other words, there is $\delta>0$ such that, if
  \[
    \max_{n =0}^{\mclN-1} \tau_n < \delta,
  \]
  then, for every $n \in \{0, \ldots, \mclN-1\}$ and all $t_1, t_2 \in I_n$, we have
  \[
    \| f_\varepsilon(t_1) - f_\varepsilon(t_2) \|_{V'}^{\AJS{2}} < \varepsilon.
  \]
  Therefore,
  \begin{equation}
  \label{eq:LebesgueNumbers}
    \| f_\varepsilon - \tilde{f}_\bftau \|_{L^2(0,T;V')}^2 < \varepsilon \sum_{n=0}^{\mclN-1} \int_{I_n} |\ell_{n,i}|^2 \diff t.
  \end{equation}
  Observe now that, for $t \in I_n$,
  \[
    |\ell_{n,i}(t)| = \prod_{j=1,j\neq i}^s \frac{|t-(t_n + c_j \tau_n)|}{\tau_n |c_i - c_j|} \leq \prod_{j=1,j\neq i}^s \frac{\tau_n\max\{|c_j|,|1-c_j|\}}{\tau_n |c_i - c_j|} = \prod_{j=1,j\neq i}^s \frac{\max\{|c_j|,|1-c_j|\}}{|c_i - c_j|} \leq C,
  \]
  where this constant is independent of $\bftau$ and $\varepsilon$, as it only depends on the entries of the Butcher table. With this at hand we continue \eqref{eq:LebesgueNumbers} as
  \[
    \| f_\varepsilon - \tilde{f}_\bftau \|_{L^2(0,T;V')}^2 < C^2 \varepsilon \sum_{n=0}^{\mclN-1} \int_{I_n} \diff t = C^2T \varepsilon,
  \]
  which clearly implies convergence.
\end{proof}

\begin{remark}[choice of $f_\varepsilon$]
  In the context of Proposition~\ref{prop:FApprox}, suitable choices of $f_\varepsilon$ are, for instance, a globally continuous quasiinterpolant \cite{MR3702417}, or simply obtained by some sort of regularization like convolution with a standard mollifier.
\end{remark}

\subsection{Convergence}

Let us now show that a DIRK scheme that satisfies Proposition~\ref{prop:BochnerStable} is convergent. First of all, on the basis of Proposition~\ref{prop:BochnerStable}, we extract convergent subsequences.

%

\begin{proposition}[limit passage]
\label{prop:Limit}
  Let $s \in \{2,3\}$. Assume that the $s$-stage DIRK scheme \eqref{eq:Stages}--\eqref{eq:Advance} is consistent to order $2s-2$ and $U$-stable in the sense of Definition~\ref{def:Ustability}. Assume also that the sequence $\{f_{n,i}\}_{i=1,n=0}^{s,\mclN-1}$ is defined as in Proposition~\ref{prop:FApprox}. Then, for any family of partitions we have that
  \begin{equation}
  \label{eq:EnergyApriori}
    \nrm{\bar{u}_\bftau}_{L^\infty(0,T;H)}^2 + \nrm{\tilde{v}_\bftau}_{L^2(0,T;V)}^2 \lesssim \nrm{u_0}_H^2 + \nrm{ f }_{L^2(0,T;V')}^2,
  \end{equation}
  and
  \begin{equation}
  \label{eq:DerivApriori}
    \nrm{\diff_t \hat{u}_\bftau }_{L^2(0,T;V')} \lesssim \nrm{u_0}_H + \nrm{ f }_{L^2(0,T;V')},
  \end{equation}
  where the implied constants are independent of $\bftau$. Consequently, there is $u \in \mclW$ such that, as $\bftau \to \bfzero$ and up to subsequences,
  \begin{equation}
  \label{eq:limits}
    \begin{aligned}
      \bar{u}_\bftau &\rightharpoonup^* u, & \text{ in } L^\infty(0,T;H), &&
      \hat{u}_\bftau &\rightharpoonup^* u, & \text{ in } L^\infty(0,T;H), \\
      \tilde{v}_\bftau &\rightharpoonup u, & \text{ in } L^2(0,T;V), &&
      \diff_t\hat{u}_\bftau &\rightharpoonup \diff_t u, & \text{ in } L^2(0,T;V').
    \end{aligned}
  \end{equation}
\end{proposition}
\begin{proof}
  \AJS{We begin by recalling that, given our standing assumptions, we have estimate \eqref{eq:BochnerStableGlobal}.}
%
%
  Next, for every $n$, $\tilde{v}_{\bftau|\mathring{I}_n} \in \mbbP_{s-1} \otimes V$ and $\tilde{f}_{\bftau|\mathring{I}_n} \in \mbbP_{s-1} \otimes V'$, so that $\nrm{\tilde{v}_{\bftau|\mathring{I}_n}}_V^2,\nrm{ \tilde{f}_{\bftau|\mathring{I}_n} }_{V'}^2 \in \mbbP_{2s-2}$. Then, since we have assumed that the DIRK scheme is consistent to order $2(s-1)$, Lemma~\ref{lem:quadrature} implies that
  \begin{align*}
    \sum_{n=1}^{\mclN} \tau_n \sum_{i=1}^s b_i \nrm{ v_{n,i} }_V^2 &= \int_0^T \nrm{ \tilde{v}_\bftau }_V^2 \diff t,
    \\
    \sum_{n=1}^{\mclN} \tau_n \sum_{i=1}^s b_i \nrm{ f_{n,i} }_{V'}^2 &= \int_0^T \nrm{ \tilde{f}_\bftau }_{V'}^2 \diff t \lesssim  \| f \|_{L^2(0,T;V')}^2,
  \end{align*}
  where, to obtain the upper bound, we invoked Proposition~\ref{prop:FApprox}. 
  Since $\bar{u}_\bftau$ is piecewise constant in time, the previous considerations imply \eqref{eq:EnergyApriori}. \AJS{Similarly, \eqref{eq:DerivApriori} is a rewriting of \eqref{eq:DualNormBound}.}

  Having obtained \eqref{eq:EnergyApriori} and \eqref{eq:DerivApriori} we may extract a (not relabeled) subsequence and find $\bar{u},\hat{u} \in L^\infty(0,T;H)$, $u \in L^2(0,T;V)$, and $w \in L^2(0,T;V')$ such that
  \begin{align*}
    \bar{u}_\bftau &\rightharpoonup^* \bar{u}, & \text{ in } L^\infty(0,T;H), &&
    \hat{u}_\bftau &\rightharpoonup^* \hat{u}, & \text{ in } L^\infty(0,T;H), \\
    \tilde{v}_\bftau &\rightharpoonup u, & \text{ in } L^2(0,T;V), &&
    \diff_t\hat{u}_\bftau &\rightharpoonup w, & \text{ in } L^2(0,T;V'),
  \end{align*}
  where we also used the fact that $\nrm{\hat{u}_\bftau }_{L^\infty(0,T;H)} = \nrm{\bar{u}_\bftau }_{L^\infty(0,T;H)}$. We must now show that, in fact, $u = \bar{u} = \hat{u}$ and $\diff_t u = w$.

  To achieve this, we first show that $\tilde{v}_\bftau - \hat{u}_\bftau \to 0$ in $L^2(0,T;V')$. Let $n$ be arbitrary and we recall the extrapolation identity \eqref{eq:Extrapolation} to write, for $t \in \mathring{I}_n$,
  \[
    \hat{u}_\bftau(t) - \tilde{v}_\bftau(t) = \sum_{i=1}^s (v_{n,i} - u_n) \left( \lambda_i (1-\alpha_n(t)) - \ell_{n,i}(t) \right),
  \]
  where we have denoted $\alpha_n(t) = \tfrac{t -t_n}{\tau_n}$. Let now $\psi \in L^2(0,T)$ be arbitrary and compute
  \begin{align*}
    \mclC_n(\psi) &\coloneqq \left| \int_{I_n} \psi(t) \left( \lambda_i (1-\alpha_n(t)) - \ell_{n,i}(t) \right) \diff t \right| \\
      &\leq \nrm{ \psi }_{L^2(I_n)} \left( \int_{I_n} \left( \lambda_i (1-\alpha_n(t)) - \ell_{n,i}(t) \right)^2 \diff t \right)^{1/2}.
  \end{align*}
  A simple change of variable reveals that
  \[
    \int_{I_n} \left| \lambda_i (1-\alpha_n(t)) - \ell_{n,i}(t) \right|^2 \diff t = \tau_n \int_0^1 \left|\lambda_i (1-r) - \prod_{j\neq i} \frac{r-c_j}{c_i-c_j} \right|^2 \diff r,
  \]
  and so for a constant $M$, that depends only on the entries of the Butcher table,
  \[
    \mclC_n(\psi) \leq M \tau_n^{1/2} \nrm{ \psi }_{L^2(I_n)}.
  \]
  Next, for $\phi \in V$  and $\psi \in L^2(0,T)$,
  \begin{align*}
    \left|\int_0^T \langle \hat{u}_\bftau - \tilde{v}_\bftau, \phi \psi \rangle \diff t \right| &\leq \sum_{n=0}^{\mclN-1} \sum_{i=1}^s \mclC_n(\psi) | \langle v_{n,i} - u_n, \phi \rangle|
    \\
    &\leq M \max_{n=0}^{\mclN-1} \tau_n^{1/2} \sum_{n=0}^{\mclN-1} \tau_n \nrm{\psi}_{L^2(I_n)} \sum_{i=1}^s \sum_{j=1}^i |a_{ij}| \left| \langle f_{n,j} - \mclA v_{n,j}, \phi \rangle \right|
    \\
    &\leq
    \check{M} \max_{n=0}^{\mclN-1} \sqrt{\tau_n} \sum_{n=0}^{\mclN-1} \tau_n \nrm{\psi}_{L^2(I_n)} \nrm{\phi}_V \sum_{i=1}^s \left[ \nrm{f_{n,i} }_{V'} + \nrm{\mclA v_{n,i} }_{V'} \right],
    \\
    \leq &\check{M} \max_{n=0}^{\mclN-1} \tau_n \nrm{\psi}_{L^2(0,T)} \nrm{\phi}_V \mclX,
  \end{align*}
  with $\check{M} = M \max_{i=1}^s \sum_{j=1}^i |a_{ij}|$ and
  \[
    \mclX = \left( \sum_{n=0}^{\mclN-1} \tau_n \left[ \sum_{i=1}^s \left[ \nrm{f_{n,i} }_{V'} + \nrm{\mclA v_{n,i} }_{V'} \right] \right]^2 \right)^{1/2}.
  \]
  Now, since $\{b_i\}_{i=1}^s \subset (0,1)$, letting $B = \min_{i=1}^s b_i^{1/2}$ we bound $\mclX$ as
  \[
    \mclX \leq \frac{s^{1/2}} B \left[ \sum_{n=0}^{\mclN-1} \tau_n \sum_{i=1}^s b_i  \left( \nrm{f_{n,i} }_{V'}^2 + \nrm{\mclA v_{n,i} }_{V'}^2 \right)  \right]^{1/2}
    \lesssim \left( \nrm{f}_{L^2(0,T;V')}^2 + \nrm{ \tilde{v}_\bftau }_{L^2(0,T;V)}^2 \right)^{1/2},
  \]
  where we also used the boundedness of $\mclA$ and the convergence of Proposition~\ref{prop:FApprox}. Since linear combinations of terms of the form $\phi\psi$ are dense in $L^2(0,T;V)$ the limit follows.

  We now show that $\tilde{v}_\bftau - \bar{u}_\bftau \to 0$ in $L^2(0,T;V')$ as well. We use, once again, the extrapolation identity \eqref{eq:Extrapolation} to see that, for any $n$ and all $t \in \mathring{I}_n$,
  \[
    \tilde{v}_\bftau(t) - \bar{u}_\bftau(t) = \sum_{i=1}^s (v_{n,i} - u_n) (\ell_{n,i}(t) - \lambda_i).
  \]
  A similar argument as before shows that, for $\psi \in L^2(0,T)$,
  \[
    \mclD_n \coloneqq \left| \int_{I_n} \psi(t) (\ell_{n,i}(t) - \lambda_i) \diff t \right| \leq M \tau_n^{1/2} \nrm{\psi}_{L^2(I_n)}.
  \]
  Let now $\phi \in V$. Similar computations to the ones presented above give
  \begin{align*}
    \left|\int_0^T \langle \tilde{v}_\bftau - \bar{u}_\bftau, \phi \psi \rangle \diff t \right| &\leq \check{M} \max_{n=0}^{\mclN-1} \tau_n \nrm{\psi}_{L^2(0,T)} \nrm{\phi}_V \left( \nrm{f}_{L^2(0,T;V')}^2 + \nrm{ \tilde{v}_\bftau }_{L^2(0,T;V)}^2 \right)^{1/2}.
  \end{align*}

  The previous limits show that $u = \hat{u} = \bar{u}$ in $L^2(0,T;V')$.

  Finally, for $\varphi \in C_0^\infty(0,T,V)$, we consider, as $\bftau \to \bfzero$,
  \[
    \int_0^T \langle \diff_t \hat{u}_\bftau, \varphi \rangle \diff t = -\int_0^T ( \hat{u}_\bftau, \diff_t \varphi ) \diff t \to -\int_0^T ( u, \diff_t \varphi ) \diff t = \int_0^T \langle \diff_t u, \varphi \rangle \diff t.
  \]
  On the other hand,
  \[
    \int_0^T \langle \diff_t \hat{u}_\bftau, \varphi \rangle \diff t \to \int_0^T \langle w, \varphi \rangle \diff t.
  \]
  Consequently, $\diff_t u = w \in L^2(0,T;V')$.

  All the statements have been proved.
\end{proof}

The convergence result now follows.

\begin{theorem}[convergence]
\label{thm:convergence}
  In the setting of Proposition~\ref{prop:Limit} we have that the function $u \in \mclW$ solves \eqref{eq:VarSolution}. In addition, for any partition, as $\bftau \to \bfzero$, we have that the whole family satisfies \eqref{eq:limits}.
\end{theorem}
\begin{proof}
  The convergence of the whole family from subsequence convergence is due to uniqueness of solutions to \eqref{eq:VarSolution} and the so-called Urysohn's subsequence principle. Thus, it remains to prove that the limit $u$ solves \eqref{eq:VarSolution}.

  First, since $\hat{u}_\bftau(0) = u_0$, and $\hat{u}_\bftau \rightharpoonup u$ in $\mclW \hookrightarrow C([0,T];H)$, we have $u(0) = u_0$.

  Next, we let $w \in L^2(0,T;V)$ and define, for $n \in \{0, \ldots, \mclN-1\}$,
  \[
    w_n \coloneqq \frac1{\tau_n} \int_{I_n} w \diff t \in V.
  \]
  It is not difficult to show that, if $\chi_n$ denotes the characteristic of $\mathring{I}_n$ then, as $\bftau \to \bfzero$,
  \[
    \sum_{n=0}^{\mclN-1} w_n \chi_n \to w
  \]
  in $L^2(0,T;V)$. We may use \eqref{eq:AdvaceIntLoc} with $w=w_n$ to write
  \[
    \int_0^T \left\langle \diff_t \hat{u}_\bftau(t) + \mclA \tilde{v}_\bftau(t) - \tilde{f}_\bftau(t), \sum_{n=0}^{\mclN-1} \chi_{n}(t) w_n \right\rangle \diff t = 0.
  \]
  Since Propositions~\ref{prop:Limit} and~\ref{prop:FApprox} show that
  \[
    \diff_t \hat{u}_\bftau + \mclA \tilde{v}_\bftau - \tilde{f}_\bftau \rightharpoonup \diff_t u + \mclA u - f
  \]
  in $L^2(0,T;V')$, the result follows.
\end{proof}

\section{Some remarks on DIRK schemes for gradient flows}
\label{sec:GradFlows}

The study of DIRK schemes for gradient flows of the form \eqref{eq:EVI} is carried out, for instance, in \cite{MR2566077}; see also \cite{MR4159230,MR4007568}. The following result is \cite[Theorem 1]{MR2566077}.

\begin{theorem}[convergence]
  Let the DIRK scheme \eqref{eq:Stages}--\eqref{eq:Advance} be consistent to order at least one and $B$-stable. Assume that the time-step is uniform, that is, for $\mclN \in \Natural$, we have
  \[
    \tau \coloneqq \frac{T}\mclN, \qquad t_n = n \tau, \quad n = 0, \ldots, \mclN.
  \]
  Then, the piecewise constant reconstruction $\bar{u}_\bftau \in L^\infty(0,T;H)$ obtained from applying the DIRK scheme to approximate \eqref{eq:EVI} satisfies
  \[
    \lim_{\bftau \to \bfzero} \| u - \bar{u}_\bftau \|_{L^\infty(0,T;H)} = 0.
  \]
\end{theorem}

We shall provide some further properties of DIRK schemes that are $U$- or $B$-stable when applied to \eqref{eq:EVI}.
We begin by recalling that gradient flows are nonexpansive, in the sense that if $u, \tilde{u}$ denote solutions to \eqref{eq:EVI}, then
\[
  \| u(t) - \tilde{u}(t) \|_H \leq \| u(0) - \tilde{u}(0) \|_H, \qquad \forall t \in (0,T].
\]
In \cite{MR4121887} the question of contractivity of Runge-Kutta schemes applied to convex gradient flows where the derivative of the potential is globally Lipschitz continuous is studied. It is shown that the contractive property of the exact flow does not automatically transfer to arbitrary Runge-Kutta discretizations. The construction of a convex potential for which there is a Runge-Kutta scheme that is not contractive for any time step is demonstrated as well.

\subsection{Discrete energy dissipation}

Our goal here is to understand the behavior of the energy as we move from $u_n$ to the stages, from one stage to the next, and to $u_{n+1}$. To this end we begin by recalling that, according to \cite[Theorem 2.1]{MR3194795}, see also \cite{MR1293524}, if the energy $\Phi$ is twice continuously differentiable and with Lipschitz continuous Hessian, and the time step is sufficiently small, we have
\begin{equation}
\label{eq:Phiunp1leqPhiun}
  \Phi(u_{n+1}) \leq \Phi(u_n), \qquad \forall n \in \{0, \ldots, \mclN-1\}.
\end{equation}
In applications of gradient flows, in particular those related to partial differential equations, energies are rarely twice differentiable, and when they are, the time step constraint involves the Lipschitz constant of the Hessian, making this restriction rather stringent in practice. For this reason, here we explore some alternatives.

%

We begin with a variational characterization of the stages.

\begin{lemma}[variational characterization]
\label{lem:VarStages}
  Assume that the DIRK scheme \eqref{eq:Stages}--\eqref{eq:Advance} is $U$-stable in the sense of Definition~\ref{def:Ustability}. Define, for $n \in \{0, \ldots, \mclN-1\}$,
  \begin{align*}
    g_{n,1} &\coloneqq u_n, 
    \\
    g_{n,i} &\coloneqq u_n - \sum_{j=1}^{i-1} a_{ij}\tau_n \Diff\Phi(v_{n,j}).
  \end{align*}
  Then, the stages satisfy
  \[
    \frac1{2a_{ii}\tau_n} \left[ \| v_{n,i} - g_{n,i} \|_H^2 - \| w - g_{n,i} \|_H^2 \right] + \Phi(v_{n,i}) \leq \Phi(w), \quad \forall w \in H.
  \]
  In particular, we have
  \[
    \Phi(v_{n,1}) \leq \Phi(u_n).
  \]
\end{lemma}
\begin{proof}
  The sought variational characterization is obtained by realizing that each stage can be written as a minimizing-movements-like scheme. Thus, we follow ideas from this context; see, for instance, \cite[Lecture 12]{MR4886890}.
  
  Notice that, in our setting, we may rewrite \eqref{eq:Stages} as
  \[
    \frac1{a_{ii}\tau_n} \left( v_{n,i} - g_{n,i}, w \right) + \Diff \Phi(v_{n,i}) = 0.
  \]
  In other words, each stage minimizes the energy
  \[
    \Phi_{n,i}(w) \coloneqq \frac1{2a_{ii}\tau_n} \left\| w - g_{n,i} \right\|_H^2 + \Phi(w),
  \]
  that is,
  \[
    \Phi_{n,i}(v_{n,i}) \leq \Phi_{n,i}(w), \qquad \forall w\in H,
  \]
  which is the claimed inequality.
  
  To bound the energy at the first stage, we only need to set $w=u_n$ in the obtained estimate and use the definition of $g_{n,1}$.
\end{proof}

We next obtain what we consider a discrete analogue of \eqref{eq:EnergyDissipation}. To state it, we shall define
\begin{equation}
\label{eq:defOfLocalMathringV}
  \mathring{v}_n \coloneqq \sum_{i=1}^s b_i v_{n,i}.
\end{equation}

\begin{proposition}[local energy dissipation]
\label{prop:LocEnergyDissipationGradFlows}
  Assume that the DIRK scheme \eqref{eq:Stages}--\eqref{eq:Advance} is $U$-stable in the sense of Definition~\ref{def:Ustability}. Then, for every $n \in \{ 0, \ldots, \mclN-1\}$, we have
  \[
    \frac1{2\tau_n} \left[ \| u_{n+1} - w\|_H^2 - \| u_n - w \|_{\AJS{H}}^2 \right] + \Phi(\mathring{v}_n) \leq \Phi(w), \qquad \forall w \in H.
  \]
\end{proposition}
\begin{proof}
  By convexity, for every $w \in H$, we have
  \[
    \left( \Diff\Phi(v_{n,i}), v_{n,i}-w \right) \geq \Phi(v_{n,i}) - \Phi(w),
  \]
  which we may multiply by $b_i>0$ and add over $i=1,\ldots,s$ to obtain
  \[
    \sum_{i=1}^s b_i \left( \Diff\Phi(v_{n,i}), v_{n,i}-w \right) \geq \Phi( \mathring{v}_n ) - \Phi(w).
  \]
  With this at hand \AJS{we notice that \eqref{eq:BochnerStableLocal},} when made specific to our setting, can be rewritten as
  \[
    \frac1{2\tau_n}\| u_{n+1}\|_H^2 + \sum_{i=1}^s b_i \left( \Diff \Phi(v_{n,i}), w \right) + \Phi( \mathring{v}_n ) \leq \frac1{2\tau_n}\| u_{n}\|_H^2 + \Phi(w),
  \]
  for every $w \in H$. Next, we take the inner product of the update equation \eqref{eq:Advance} with $w$ to obtain
  \[
    \sum_{i=1}^s b_i \left( \Diff \Phi(v_{n,i}), w \right) = -\frac1{\tau_n} \left( u_{n+1} - u_n,w \right).
  \]
  Since
  \[
    \| u_{n+1}\|_H^2 - 2\left( u_{n+1} - u_n,w \right) - \| u_{n}\|_H^2 = \| u_{n+1} - w \|_{\AJS{H}}^2 - \| u_n - w\|_{\AJS{H}}^2.
  \]
  the result follows.
\end{proof}

We state an important consequence of the previous estimate.

\begin{corollary}[local energy dissipation]
\label{cor:LocalEnergyDissipation}
  In the setting of Proposition~\ref{prop:LocEnergyDissipationGradFlows}, we have
  \[
    \frac{\tau_n}2 \left\| \frac{u_{n+1}-u_n}{\tau_n} \right\|_H^2 + \Phi(\mathring{v}_{n}) \leq \Phi(u_n), \qquad n = 0, \ldots, \mclN-1,
  \]
\end{corollary}
\begin{proof}
  One merely needs to set $w = u_n$.
\end{proof}


\subsection{About energy monotonicity of the stages}

We now make some comments with regard to Corollary~\ref{cor:LocalEnergyDissipation}. First, while we have
\[
  \Phi(\mathring{v}_n) \leq \Phi(u_n),
\]
a stronger form of local energy dissipation would be that, for all $n$ and $i<j$,
\[
  \Phi(v_{n,j}) \leq \Phi( v_{n,i} ).
\]
The following simple counterexample, however, shows that this is not the case in general.

We introduce the two stage scheme
\[
  \begin{array}{c|cc}
    \frac13 & \frac13 & 0
    \\
    \frac56 & \frac23 & \frac16
    \\
    \hline
           & \frac23 & \frac13
  \end{array} \ .
\]
The reader can easily verify that this scheme is consistent to second order, and that $\msfM = \msfO$, so that this scheme is $B$-stable. Consider now the simplest of examples, namely $H = \Real$ and, for $k>0$,
\begin{equation}
\label{eq:SimplestEnergy}
  \Phi(w) = \frac{k}2 w^2.
\end{equation}
For any $u_n \in \Real$,
\[
  v_{n,1} = \frac{3}{3+k\tau_n} u_n,
  \qquad\qquad
  v_{n,2} = \frac{6(3-k\tau_n)}{(3+k\tau_n)(6+k\tau_n)} u_n,
  \qquad\qquad
  u_{n+1} = \frac{(3-k\tau_n)(6-k\tau_n)}{(3+k\tau_n)(6+k\tau_n)} u_n.
\]
Clearly, for every value of $k$ and $\tau_n$,
\begin{equation}
\label{eq:PhivleqPhiun}
  \Phi( v_{n,i} ) \leq \Phi(u_n), \qquad i = 1,2,
\end{equation}
and \eqref{eq:Phiunp1leqPhiun} holds. However, if $k \tau_n > 12$,
\[
  \Phi(v_{n,2}) > \Phi(v_{n,1}),
\]
unless $u_n = 0$. On the other hand,
\[
  \mathring{v}_n = \frac{18}{(3+k\tau_n)(6+k\tau_n)} u_n,
\]
which clearly satisfies $\Phi(\mathring{v}_n) \leq \Phi(u_n)$. However, if $k \tau_n >9$,
\[
  \Phi(u_{n+1}) > \Phi\left( \mathring{v}_n \right).
\]

In the previous example the energy of the stages was controlled by the energy at $u_n$; see \eqref{eq:PhivleqPhiun}. The final example, however, shows that this is not the case in general either.

The DIRK scheme
\[
  \begin{array}{c|cc}
    \gamma & \gamma \\
    1-\gamma & 1-2\gamma & \gamma
    \\
    \hline
    & \frac12 & \frac12
  \end{array}
  \qquad\qquad\qquad
  \gamma = 1 + \frac{\sqrt{2}}2
\]
was originally introduced in \cite{MR518683}. It is $U$-stable; see \cite[Section 4.2.1]{MR4673496}.

We apply the scheme to the gradient flow with energy given by \eqref{eq:SimplestEnergy}. For $u_n \in \Real$ we obtain
\[
  v_{n,1} = \frac{2}{ 2 + \left(\sqrt{2}+2\right) k\tau_n } u_n,
  \qquad
  v_{n,2} = \frac{2  \left(4 k \tau_n +3 \sqrt{2} k \tau_n +2\right)}{\left(2 k \tau_n +\sqrt{2} k \tau_n +2\right)^2} u_n,
  \qquad
  u_{n+1} = \frac{4  \left(k \tau_n +\sqrt{2} k \tau_n +1\right)}{\left(2 k \tau_n +\sqrt{2} k \tau_n +2\right)^2} u_n.
\]
Simple and direct, but somewhat lengthy, calculations show that, if
\[
  k \tau_n > 3\sqrt{2}-4 \approx 0.242641,
\]
then we have 
\[
  \Phi(v_{n,2}) > \Phi(u_n).
\]

In conclusion, even for $U$-stable schemes, without restrictions on the time step, very little can be asserted about energy monotonicity properties for the stages.

\section*{Acknowledgements}

The work of the authors is partially supported by NSF grant DMS-2409918.

\bibliographystyle{plain}
\bibliography{references}

\end{document}

%% file: AERMacros.tex
\usepackage{amssymb}
\usepackage{latexsym,amsmath,amsfonts,amscd}
\usepackage{amsthm}
\usepackage{color}
\usepackage[dvipsnames]{xcolor}
\usepackage{mathtools}

\usepackage{hyperref}

\DeclareMathOperator*{\diag}{diag}

\newcommand{\nrm}[1]{\left\| #1 \right\|}

\newcommand{\bftau}{{\boldsymbol{\tau}}}

\newcommand{\bflambda}{\mbox{\boldmath$\lambda$}}
\newcommand{\Real}{\mathbb{R}}

\newcommand{\Natural}{\mathbb{N}}

\newcommand{\msfA}{\mathsf{A}}
\newcommand{\msfB}{\mathsf{B}}

\newcommand{\msfM}{\mathsf{M}}

\newcommand{\msfO}{\mathsf{O}}

\newcommand{\msfQ}{\mathsf{Q}}

\newcommand{\msfK}{\mathsf{K}}
\newcommand{\msfL}{\mathsf{L}}

\newcommand{\mclA}{\mathcal{A}}
\newcommand{\mclK}{\mathcal{K}}
\newcommand{\mclX}{\mathcal{X}}
\newcommand{\mclW}{\mathcal{W}}

\newcommand{\mclQ}{\mathcal{Q}}

\newcommand{\mclC}{\mathcal{C}}
\newcommand{\mclN}{\mathcal{N}}
\newcommand{\mclD}{\mathcal{D}}
\newcommand{\mclM}{\mathcal{M}}
\newcommand{\bfone }{\mathbf{1}}
\newcommand{\bfzero}{\mathbf{0}}

\newcommand{\bfb}{\boldsymbol{b}}
\newcommand{\bfc}{\boldsymbol{c}}

\newcommand{\bfe}{\boldsymbol{e}}
\newcommand{\bff}{\boldsymbol{f}}

\newcommand{\bfl}{\boldsymbol{l}}

\newcommand{\bfy}{\boldsymbol{y}}
\newcommand{\bfz}{\boldsymbol{z}}

\newcommand{\mbbP}{\mathbb{P}}

\newcommand\restr[2]{{
		\left.\kern-\nulldelimiterspace 
		#1 
		\vphantom{\big|} 
		\right|_{#2} 
}}

\newtheorem{theorem}{Theorem}[section]
\newtheorem{corollary}{Corollary}[section]
\newtheorem{lemma}{Lemma}[section]
\newtheorem{proposition}{Proposition}[section]

\theoremstyle{definition}
\newtheorem{definition}{Definition}[section]

\newtheorem{remark}{Remark}[section]

\newcommand{\diff}{\mathrm{d}}
\newcommand{\Diff}{\mathrm{D}}